\documentclass[a4paper, 10pt]{article}

\usepackage[latin1]{inputenc}
\usepackage{subfigure}
\usepackage{amsmath}
\usepackage{amssymb}
\usepackage{amsthm}
\usepackage{amsfonts}
\usepackage[cyr]{aeguill}
\usepackage{xspace}
\usepackage[english]{babel}
\usepackage{graphicx}
\usepackage{float}
\usepackage{mathtools}
\usepackage{psfrag,floatflt,hyperref}
\usepackage{verbatim}
\usepackage{geometry}
\usepackage{multirow}
\usepackage[retainorgcmds]{IEEEtrantools}
\usepackage{stmaryrd}
\SetSymbolFont{stmry}{bold}{U}{stmry}{m}{n}
\usepackage{color}
\newtheorem{theorem}{Theorem}[section]
\newtheorem{lemma}[theorem]{Lemma}

\newenvironment{remark}[1][Remark.]{\begin{trivlist}
\item[\hskip \labelsep {\bfseries #1}]}{\end{trivlist}}
\newenvironment{example}[1][Example.]{\begin{trivlist}
\item[\hskip \labelsep {\bfseries #1}]}{\end{trivlist}}

\title{Asymptotics for the determinant of the combinatorial Laplacian on hypercubic lattices}
\author{Justine Louis}
\date{$30$th July $2015$}

\newcommand*{\defeq}{\mathrel{\vcenter{\baselineskip0.5ex \lineskiplimit0pt
                     \hbox{\scriptsize.}\hbox{\scriptsize.}}}=}

\begin{document}
        \maketitle

\begin{abstract}
In this paper, we compute asymptotics for the determinant of the combinatorial Laplacian on a sequence of $d$-dimensional orthotope square lattices as the number of vertices in each dimension grows at the same rate. It is related to the number of spanning trees by the well-known matrix tree theorem. Asymptotics for $2$ and $3$ component rooted spanning forests in these graphs are also derived. Moreover, we express the number of spanning trees in a $2$-dimensional square lattice in terms of the one in a $2$-dimensional discrete torus and also in the quartered Aztec diamond. As a consequence, we find an asymptotic expansion of the number of spanning trees in a subgraph of $\mathbb{Z}^2$ with a triangular boundary.
\end{abstract}
\section{Introduction}
In this paper we study the asymptotic behaviour of the number of spanning trees in a discrete $d$-dimensional orthotope square lattice and in the quartered Aztec diamond. Let $L(n_1,\ldots,n_d)$ denote the $d$-dimensional orthotope square lattice defined by the cartesian product of the $d$ path graphs $P_{n_i}$, $i=1,\ldots,d$, where $n_i$, $i=1,\ldots,d$, are positive non-zero integers. We set $n_i=\alpha_in$, $i=1,\ldots,d$ and write indifferently $n_i$ or $\alpha_in$ throughout the paper. By rescaling the distance between two vertices on the lattice $L(n_1,\ldots,n_d)$ with a factor of $1/n$, the limiting object as $n$ goes to infinity is a $d$-dimensional orthotope of size $\alpha_1\times\cdots\times\alpha_d$, that we denote by $K_d$: 
\begin{equation*}
K_d\defeq\left[0,\alpha_1\right[\times\cdots\times\left[0,\alpha_d\right[.
\end{equation*}
The volume of $K_d$ is
\begin{equation*}
V_d^d\defeq\prod_{i=1}^d\alpha_i.
\end{equation*}
Let $m\in\{1,\ldots,d-1\}$. An $m$-dimensional face of $\overline{K}_d$ is defined by
\begin{align*}
\{(x_1,\ldots,x_d)\in\overline{K}_d\vert\ \exists\{i_q\}_{q=1}^d\subset\{1,\ldots,d\}\textnormal{ such that }&x_{i_q}\in\left[0,\alpha_{i_q}\right],\ q=1,\ldots,m\\
\textnormal{ and }&x_{i_q}\in\{0,\alpha_{i_q}\},\ q=m+1,\ldots,d\}.
\end{align*}
The volume of the sum of all the $m$-dimensional faces of $\overline{K}_d$ is given by
\begin{equation*}
V_m^d\defeq2^{d-m}\sum_{1\leqslant i_1<\cdots<i_m\leqslant d}\prod_{q=1}^m\alpha_{i_q}.
\end{equation*}
For example, $V^d_1$ is the perimeter and $V^d_2$ the area of $K_d$.\\
Asymptotics for the determinant of the combinatorial Laplacian on graphs have been widely studied, see for example \cite{chinta2010zeta,chinta2011complexity,MR952941,MR1819995,louis2015asymptotics}. It is related to the number of spanning trees of a graph $G$, denoted by $\tau(G)$, through the matrix tree theorem due to Kirchhoff (see \cite{MR1271140})
\begin{equation*}
\tau(G)=\frac{1}{\lvert V(G)\rvert}\textnormal{det}^\ast\Delta_G
\end{equation*}
where $\textnormal{det}^\ast\Delta_G$ is the product of the non-zero eigenvalues of the Laplacian on $G$ and $\lvert V(G)\rvert$ the number of vertices in $G$. In \cite{chinta2010zeta}, the authors developed a technique to compute the asymptotic behaviour of spectral determinants of the combinatorial Laplacian associated to a sequence of discrete tori. The technique consists in studying the asymptotic behaviour of the associated theta function which contains the spectral information of the graph. Consider a graph $G$ with vertex set $V(G)$. For a function $f$ defined on $V(G)$, the combinatorial Laplacian is defined by
\begin{equation*}
\Delta_Gf(x)=\sum_{y\sim x}(f(x)-f(y))
\end{equation*}
where the sum is over all vertices adjacent to $x$. Let $\{\lambda_k\}_k$ denote the spectrum of the Laplacian on $G$. The associated theta function is defined by
\begin{equation*}
\theta_G(t)=\sum_{k\in V(G)}e^{-\lambda_kt}.
\end{equation*}
To compute the asymptotic behaviour of spectral determinants on a sequence of $d$-orthotope square lattices, we express the associated theta function in terms of the theta function associated to the discrete torus with twice vertices at each side of the torus. This can be done because of the similarity of their spectrum. We then use the asymptotic results from \cite{chinta2010zeta}. The formula obtained relates the determinant of the Laplacian on the discrete lattice $L(n_1,\ldots,n_d)$ to the regularized determinant of the Laplacian on the rescaled limiting object, which is the real $d$-dimensional orthotope $K_d$, and to the ones on the $m$-dimensional boundary faces of $K_d$, $m=1,\ldots,d-1$. Moreover, we compute asymptotic results for the number of rooted spanning forests with $2$ and $3$ components.\\
We will prove the following theorem.
\begin{theorem}
\label{ThO}
Given positive integers $\alpha_i$, $i=1,\ldots,d$, let $\textnormal{det}^\ast\Delta_{L(\alpha_1n,\ldots,\alpha_dn)}$ be the product of the non-zero eigenvalues of the Laplacian on the $d$-dimensional orthotope square lattice $L(\alpha_1n,\ldots,\alpha_dn)$. Then as $n\rightarrow\infty$
\begin{align*}
\log\textnormal{det}^\ast\Delta_{L(\alpha_1n,\ldots,\alpha_dn)}&=c_dV^d_dn^d-\sum_{m=1}^{d-1}\frac{1}{4^{d-m}}\left(\int_0^\infty(1-e^{-4t})^{d-m}e^{-2mt}I_0(2t)^m\frac{dt}{t}\right)V^d_mn^m\\
&\ \ \ +(2-2^{1-d})\log{n}+\sum_{m=1}^d\sum_{1\leqslant i_1<\cdots<i_m\leqslant d}\log\textnormal{det}^\ast\Delta_{\alpha_{i_1}\times\cdots\times\alpha_{i_m}}\\
&\ \ \ +\frac{1}{2^d}\sum_{m=1}^dC^m_d(-1)^m\log(4m)+o(1)
\end{align*}
where $\textnormal{det}^\ast\Delta_{\alpha_{i_1}\times\cdots\times\alpha_{i_m}}$ is the regularized determinant of the Laplacian on the $m$-orthotope $\alpha_{i_1}\times\cdots\times\alpha_{i_m}$ with Dirichlet boundary conditions. The constant $c_d$ is
\begin{equation*}
c_d=\int_0^\infty(e^{-t}-e^{-2dt}I_0(2t)^d)\frac{dt}{t}
\end{equation*}
where $I_0$ is the modified $I$-Bessel function of order zero.
\end{theorem}
Notice that the limiting object $K_d$ can be decomposed in a disjoint union of $m$ orthotopes, $m\in\{0,1,\ldots,d\}$. More precisely, let $\alpha_{i_1}\times\cdots\times\alpha_{i_m}$, $\{i_q\}_{q=1}^m\subset\{1,\ldots,d\}$, denote the $m$-dimensional orthotope of side lengths $\alpha_{i_1},\ldots,\alpha_{i_m}$ which is open in $\mathbb{R}^m$. Then
\begin{equation*}
K_d=\{0\}\sqcup\bigsqcup_{m=1}^d\bigsqcup_{1\leqslant i_1<\cdots<i_m\leqslant d}\alpha_{i_1}\times\cdots\times\alpha_{i_m}.
\end{equation*}
This decomposition is reflected in the theorem by the appearance of the sum over this decomposition of the logarithm of the regularized determinant of the Laplacian on the $m$-dimensional faces, $m=1,\ldots,d$.\\
By expressing the eigenvalues of the Laplacian on the square lattice $L(n_1,n_2)$ in terms of the one on the two-dimensional discrete torus $\mathbb{Z}^2/\textnormal{diag}(2n_1,2n_2)\mathbb{Z}^2$, we derive a relation, which is stated below, between the number of spanning trees on these two lattices.
\begin{theorem}
Given positive integers $n_1,n_2$, let $\tau(L(n_1,n_2))$ denote the number of spanning trees on the rectangular square lattice $L(n_1,n_2)$ and $\tau(T(2n_1,2n_2))$ the number of spanning trees on the discrete torus $\mathbb{Z}^2/\textnormal{diag}(2n_1,2n_2)\mathbb{Z}^2$. We have
\begin{equation*}
\tau(L(n_1,n_2))=\frac{2^{5/4}\tau(T(2n_1,2n_2))^{1/4}}{(n_1n_2)^{1/4}((3+2\sqrt{2})^{n_1}-(3-2\sqrt{2})^{n_1})^{1/2}((3+2\sqrt{2})^{n_2}-(3-2\sqrt{2})^{n_2})^{1/2}}.
\end{equation*}
\end{theorem}
In \cite{MR1819995}, Kenyon computed asymptotics for spectral determinants on a simply-connected rectilinear region in $\mathbb{R}^2$. Here we compute it in the particular case of a triangular region. More precisely, we consider the quartered Aztec diamond of order $n$, denoted by $QAD_n$, which is the subgraph of $\mathbb{Z}^2$ with nearest neighbours connected induced by the vertices $(k_1,k_2)$ such that $k_1+k_2\leqslant n$ and $k_1,k_2\geqslant0$. Figure \ref{qad7} illustrates $QAD_7$. In \cite{MR2393253}, Ciucu derived a relation between the characteristic polynomials of the rectangular square lattice and of the quartered Aztec diamond using combinatorial arguments. From this one can deduce a relation for the number of spanning trees. In the second part of this work, we present an alternative approach for it. Consequently, we derive the asymptotic behaviour of it, stated in the following theorem, which shows that it is related to the regularized determinant of the Laplacian on the triangle with 
Dirichlet boundary conditions.
\begin{figure}[!ht]
\centering
\includegraphics[width=3.5cm]{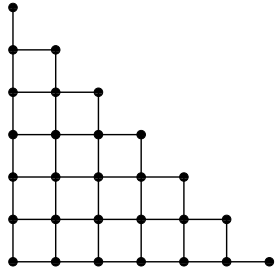}
\caption{Quartered Aztec diamond of order $7$.}
\label{qad7}
\end{figure}
\begin{theorem}
Let $\tau(\textnormal{QAD}_n)$ denote the number of spanning trees in the quartered Aztec diamond of order $n$. Then as $n\rightarrow\infty$
\begin{equation*}
\log(\tau(\textnormal{QAD}_n))=\frac{2G}{\pi}n^2-\log(2+\sqrt{2})n-\frac{3}{4}\log{n}+\log{\textnormal{det}^\ast\Delta_{\Delta}}+\frac{23}{8}\log{2}+o(1)
\end{equation*}
where $G$ is the Catalan constant and $\textnormal{det}^\ast\Delta_{\Delta}$ is the regularized determinant of the Laplacian on the right-angled isoscele unit triangle with Dirichlet boundary conditions.
\end{theorem}
\subsection{Regularized determinant}
Let $M$ be a Riemannian manifold with or without boundary and let $\Delta_M$ be the Laplace-Beltrami operator associated to $M$. If $M$ has a boundary we associate Dirichlet boundary conditions to $\Delta_M$. Denote by $\{\lambda_k\}_k$ the eigenvalues of $\Delta_M$. The spectral zeta function associated to $M$ is defined for $\Re(s)>\dim{M}/2$ by
\begin{equation*}
\zeta_M(s)=\sum_{\lambda_k\neq0}\frac{1}{\lambda_k^s}.
\end{equation*}
It admits a meromorphic continuation to the whole complex plane (see for example \cite{MR3312509} for $M$ with a boundary and \cite{chinta2010zeta} for the case $M$ is a torus). The regularized determinant of $\Delta_M$ can then be defined by
\begin{equation*}
\log\textnormal{det}^\ast\Delta_M=-\zeta'_M(0).
\end{equation*}
\subsection{Preliminary result}
We prove the following lemma which will be useful in the next section to invert relations between the theta functions. Throughout this paper, we set an empty summation to be one by convention.
\begin{lemma}
\label{inversion}
Let $\{i_q\}_{q\geqslant1}$ be an increasing sequence of positive integers. Let $\{n_q\}_{q\geqslant1}$ be a sequence of positive integers and $\{n_{i_q}\}_{q\geqslant1}$ a subsequence of it. Let $f,g:\mathbb{N}^{<\mathbb{N}}\rightarrow\mathbb{R}$ be two sequences of variadic functions such that for all $l\in\mathbb{N}_{\geqslant1}$,
\begin{equation}
\label{fgrel}
f(n_{i_1},\ldots,n_{i_l})=\sum_{k=0}^l\sum_{\substack{a_1<\cdots<a_k\\ \{a_q\}_{q=1}^k\subset\{i_q\}_{q=1}^l}}g(n_{a_1},\ldots,n_{a_k}).
\end{equation}
Then the following inversion formula holds: for all $l\in\mathbb{N}_{\geqslant1}$,
\begin{equation*}
g(n_{i_1},\ldots,n_{i_l})=\sum_{k=0}^l(-1)^{l-k}\sum_{\substack{a_1<\cdots<a_k\\ \{a_q\}_{q=1}^k\subset\{i_q\}_{q=1}^l}}f(n_{a_1},\ldots,n_{a_k}).
\end{equation*}
\end{lemma}
\begin{proof}
Let $l\in\mathbb{N}_{\geqslant1}$. From relation (\ref{fgrel}) between $f$ and $g$, we have
\begin{align}
\sum_{k=0}^l(-1)^{l-k}&\sum_{\substack{j_1<\cdots<j_k\\ \{j_q\}_{q=1}^k\subset\{i_q\}_{q=1}^l}}f(n_{j_1},\ldots,n_{j_k})\nonumber\\
&=\sum_{k=0}^l(-1)^{l-k}\sum_{\substack{j_1<\cdots<j_k\\ \{j_q\}_{q=1}^k\subset\{i_q\}_{q=1}^l}}\left(\sum_{m=0}^k\sum_{\substack{a_1<\cdots<a_m\\ \{a_q\}_{q=1}^m\subset\{j_q\}_{q=1}^k}}g(n_{a_1},\ldots,n_{a_m})\right)\nonumber\\
&=\sum_{k=1}^l(-1)^{l-k}\sum_{m=1}^k\sum_{\substack{j_1<\cdots<j_k\\ \{j_q\}_{q=1}^k\subset\{i_q\}_{q=1}^l}}\sum_{\substack{a_1<\cdots<a_m\\ \{a_q\}_{q=1}^m\subset\{j_q\}_{q=1}^k}}g(n_{a_1},\ldots,n_{a_m})
\label{doublesum}
\end{align}
where in the second equality we used that
\begin{equation*}
\sum_{k=0}^l(-1)^{l-k}\sum_{\substack{j_1<\cdots<j_k\\ \{j_q\}_{q=1}^k\subset\{i_q\}_{q=1}^l}}1=\sum_{k=0}^l(-1)^{l-k}C^k_l=0
\end{equation*}
from the binomial theorem, where $C^k_l=l!/(k!(l-k)!)$. For $l\geqslant k\geqslant m$, the double summation can be rewritten in one summation as
\begin{equation}
\label{doublesumtoonesum}
\sum_{\substack{j_1<\cdots<j_k\\ \{j_q\}_{q=1}^k\subset\{i_q\}_{q=1}^l}}\sum_{\substack{a_1<\cdots<a_m\\ \{a_q\}_{q=1}^m\subset\{j_q\}_{q=1}^k}}=C^{k-m}_{l-m}\sum_{\substack{j_1<\cdots<j_m\\ \{j_q\}_{q=1}^m\subset\{i_q\}_{q=1}^l}}.
\end{equation}
Thus (\ref{doublesum}) is equal to
\begin{align*}
\sum_{k=1}^l(-1)^{l-k}&\sum_{m=1}^kC^{k-m}_{l-m}\sum_{\substack{j_1<\cdots<j_m\\ \{j_q\}_{q=1}^m\subset\{i_q\}_{q=1}^l}}g(n_{j_1},\ldots,n_{j_m})\\
&=\sum_{m=1}^l\sum_{k=0}^{l-m}(-1)^{l-m-k}C^k_{l-m}\sum_{\substack{j_1<\cdots<j_m\\ \{j_q\}_{q=1}^m\subset\{i_q\}_{q=1}^l}}g(n_{j_1},\ldots,n_{j_m})=g(n_{i_1},\ldots,n_{i_l})
\end{align*}
where the second equality comes from the fact that the only non-zero term in the summation over $m$ is when $m=l$ from the binomial theorem. 
\end{proof}
\begin{remark}
If we set $f(n_{i_1},\ldots,n_{i_l})=f_l$ and $g(n_{i_1},\ldots,n_{i_k})=g_k$ for all $l,k\in\mathbb{N}_{\geqslant1}$, we recover the standard binomial inversion:
\begin{equation*}
f_l=\sum_{k=0}^lC^k_lg_k\textnormal{ if and only if }g_l=\sum_{k=0}^l(-1)^{l-k}C^k_lf_k,\textnormal{ for all }l\in\mathbb{N}_{\geqslant1},
\end{equation*}
with $f_0=g_0=1$.
\end{remark}
\section{\texorpdfstring{Asymptotic number of spanning trees in the \boldmath$d$-orthotope lattice}{Asymptotic number of spanning trees in the d-orthotope lattice}}
\subsection{Theta function}
The eigenvalues of the Laplacian on the square lattice $L(n_1,\ldots,n_d)$ are given by (see \cite{MR952941})
\begin{equation*}
\{\lambda_k^L\}_{k=0,1,\ldots,N-1}=\{2d-2\sum_{i=1}^d\cos(\pi k_i/n_i),\ k_i=0,1,\ldots,n_i-1,\textnormal{ for }i=1,\ldots,d\}
\end{equation*}
where $N=\prod_{i=1}^dn_i$. The $d$-dimensional discrete torus of size $2n_1\times\cdots\times2n_d$ is defined by the quotient $\mathbb{Z}^d/\textnormal{diag}(2n_1,\ldots,2n_d)\mathbb{Z}^d$ with nearest neighbours connected. We denote it by $T(2n_1,\ldots,2n_d)$. The eigenvalues of the Laplacian on $T(2n_1,\ldots,2n_d)$ are given by (see \cite{louis2015formula})
\begin{equation*}
\{\lambda_k^T\}_{k=0,1,\ldots,2^dN-1}=\{2d-2\sum_{i=1}^d\cos(\pi k_i/n_i),\ k_i=0,1,\ldots,2n_i-1,\textnormal{ for }i=1,\ldots,d\}.
\end{equation*}
Notice that $\{\lambda_k^L\}_k\subset\{\lambda_k^T\}_k$.\\
The theta function on the $d$-orthotope lattice $L(n_1,\ldots,n_d)$ is given by
\begin{equation*}
\theta_{L(n_1,\ldots,n_d)}(t)=\sum_{k_1=0}^{n_1-1}\cdots\sum_{k_d=0}^{n_d-1}e^{-(2d-2\sum_{i=1}^d\cos(\pi k_i/n_i))t}
\end{equation*}
and on the discete torus $T(2n_1,\ldots,2n_d)$ by
\begin{equation*}
\theta_{T(2n_1,\ldots,2n_d)}(t)=\sum_{k_1=0}^{2n_1-1}\cdots\sum_{k_d=0}^{2n_d-1}e^{-(2d-2\sum_{i=1}^d\cos(\pi k_i/n_i))t}.
\end{equation*}
Therefore by expressing the theta function on the $d$-orthotope square lattice $L(n_1,\ldots,n_d)$ in terms of the one on the torus $T(2n_1,\ldots,2n_d)$, one can deduce the asymptotic behaviour from the results obtained in \cite{chinta2010zeta}.\\
Let $\llbracket a,b\rrbracket$ denote the set of successive integers $\{a,a+1,\ldots,b\}$. In the theta function on $L(n_1,\ldots,n_d)$, the summation is over the discrete $d$-orthotope $\llbracket0,n_1-1\rrbracket\times\cdots\times\llbracket0,n_d-1\rrbracket$ that we denote by $J_d$, while for the torus $T(2n_1,\ldots,2n_d)$ it is over the discrete $d$-orthotope $\llbracket0,2n_1-1\rrbracket\times\cdots\times\llbracket0,2n_d-1\rrbracket$, denoted by $\widetilde{J}_d$. We decompose $J_d$ as a disjoint union of $l$-dimensional faces, $l=0,1,\ldots,d$. The $0$ dimension is the point $0\in\mathbb{Z}^d$, we call it the root of $J_d$. For $l\in\{1,\ldots,d-1\}$, the $l$-dimensional faces are defined by a subset of $\mathbb{Z}^d$, $(k_1,\ldots,k_d)\subset J_d$, such that $\exists\{i_q\}_{q=1}^d\subset\{1,\ldots,d\}$ such that $k_{i_q}\in\llbracket1,n_{i_q}-1\rrbracket$, $q=1,\ldots,l$ and $k_{i_q}=0$, $q=l+1,\ldots,d$. We call the $d$-dimensional face the interior of $J_d$ where no coordinate is zero, that is $\llbracket1,n_1-
1\rrbracket\times\cdots\times\llbracket1,n_d-1\rrbracket$. For example, in the $2$ dimensional case, $J_2$ decomposes as:
\begin{equation*}
J_2=\{0\}\sqcup(\llbracket1,n_1-1\rrbracket\times\{0\})\sqcup(\{0\}\times\llbracket1,n_2-1\rrbracket)\sqcup(\llbracket1,n_1-1\rrbracket\times\llbracket1,n_2-1\rrbracket).
\end{equation*}
In the theta function of the torus $T(2n_1,\ldots,2n_d)$, the summation is over $2^d$ copies of $J_d$, namely:
\begin{equation*}
\widetilde{J}_d=\bigsqcup_{\substack{\epsilon_i\in\{0,n_i\}\\i=1,\ldots,d}}(J_d+(\epsilon_1,\ldots,\epsilon_d))
\end{equation*}
where the unions are disjoint. Figure \ref{J2} illustrates the decompositions of $J_d$ and $\widetilde{J}_d$ in the case $d=2$.
\begin{figure}[H]
\centering
\subfigure[$J_2$]{\includegraphics[width=2.5cm]{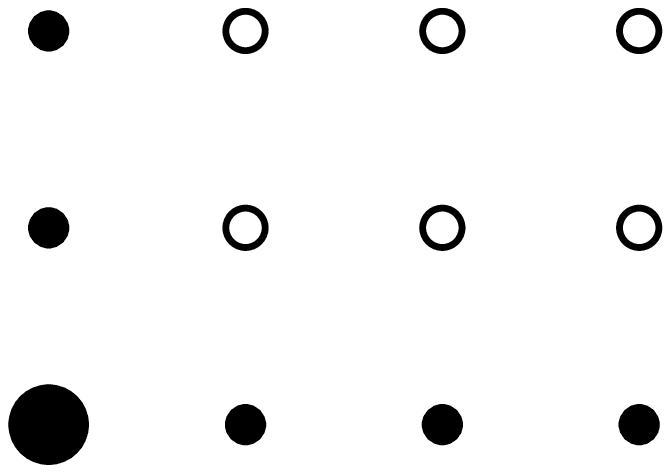}}
\hspace{2cm}
\subfigure[$\widetilde{J}_2$]{\includegraphics[width=5cm]{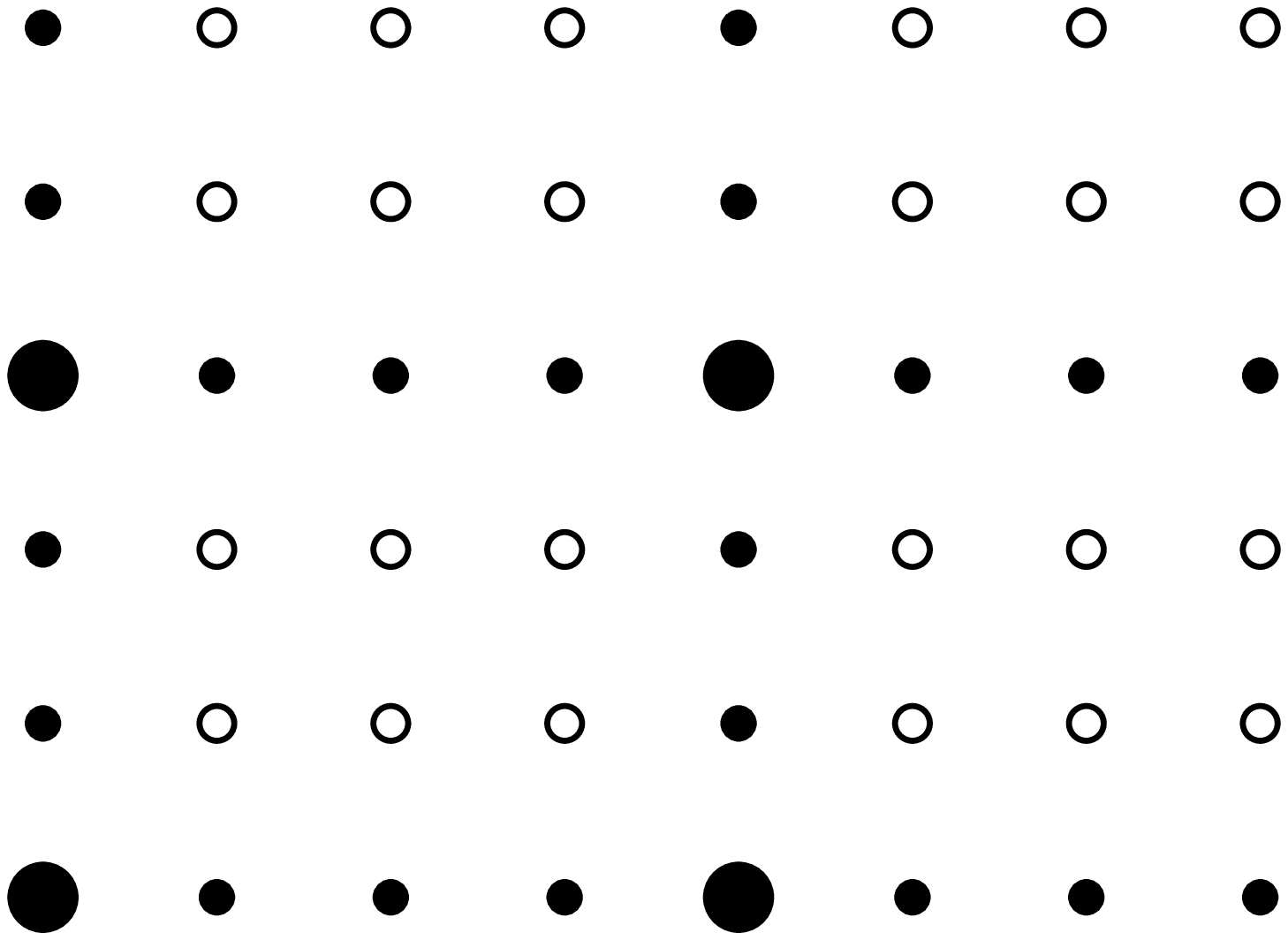}}
\caption{$J_2$ and $\widetilde{J}_2$ with $n_1=4$ and $n_2=3$. The big black dot is the root of $J_2$, the small black dots are the two $1$-dimensional faces and the white dots are the interior of $J_2$.}
\label{J2}
\end{figure}
Let define the theta-star function on the lattice $L(n_{i_1},\ldots,n_{i_l})$, with $1\leqslant i_1<\cdots<i_l\leqslant d$, by the following expression
\begin{equation*}
\theta_{L(n_{i_1},\ldots,n_{i_l})}^\ast(t)=\sum_{k_{i_1}=1}^{n_{i_1}-1}\cdots\sum_{k_{i_l}=1}^{n_{i_l}-1}e^{-(2l-2\sum_{i=i_1}^{i_l}\cos(\pi k_i/n_i))t}
\end{equation*}
where the summation is over the interior of the $l$-orthotope of size $n_{i_1}\times\cdots\times n_{i_l}$, that is, the $k_i$ start at $1$ instead of $0$ for all $i\in\{i_1,\ldots,i_l\}$. Therefore the theta function is related to the theta-star function by the relation
\begin{equation}
\label{thetaLL*}
\theta_{L(n_1,\ldots,n_d)}(t)=\sum_{l=0}^d\sum_{1\leqslant i_1<\cdots<i_l\leqslant d}\theta_{L(n_{i_1},\ldots,n_{i_l})}^\ast(t).
\end{equation}
To evaluate the theta function on $T(2n_1,\ldots,2n_d)$ we sum over the roots of the $2^d$ discrete orthotopes $J_d$, the $l$-dimensional boundary faces of $J_d$ and the interior of $J_d$. Summing over the roots of the $J_d$'s means that for all $i=1,\ldots,d$, $k_i$ is either $0$ either $n_i$. There are $C^j_d=d!/(j!(d-j)!)$ number of ways to take $j$ of the $k_i$'s to be zero. In this case the corresponding exponential term of the theta function is $e^{-4(d-j)t}$. The term $l=0$ is then the sum of all the possibilities:
\begin{equation*}
\sum_{j=0}^dC^j_de^{-4(d-j)t}=(1+e^{-4t})^d.
\end{equation*}
Then for each $l$-dimensional boundary face of $J_d$, where $l\in\{1,\ldots,d-1\}$, there are $d-l$ of the $k_i$'s which are either $0$ either $n_i$, this positions the $l$-dimensional face in $\widetilde{J}_d$. For this $l$, there are $C^j_{d-l}$ number of ways where $j$ of the $k_i$'s are zero and the exponential term is then $e^{-4(d-l-j)t}$. And we sum over the interior of the $l$-dimensional face $L(n_{i_1},\ldots,n_{i_l})$, where $1\leqslant i_1<\cdots<i_l\leqslant d$, which is by definition the theta-star function $\theta_{L(n_{i_1},\ldots,n_{i_l})}^\ast(t)$ with a factor of $2^l$ since this configuration appears $2^l$ times. So for $l\in\{1,\ldots,d-1\}$, we have
\begin{equation*}
\sum_{j=0}^{d-l}C_{d-l}^je^{-4(d-l-j)t}\sum_{1\leqslant i_1<\cdots<i_l\leqslant d}2^l\theta_{L(n_{i_1},\ldots,n_{i_l})}^\ast(t).
\end{equation*}
Finally we sum over the interior of $J_d$, that is, when all the $k_i$'s are greater or equal to one, which appears $2^d$ times. This gives the $l=d$ term
\begin{equation*}
2^d\theta_{L(n_1,\ldots,n_d)}^\ast(t).
\end{equation*}
The theta function on $T(2n_1,\ldots,2n_d)$ is then the sum over the $l$-dimensional faces of the $2^d$ orthotopes $J_d$:
\begin{equation*}
\theta_{T(2n_1,\ldots,2n_d)}(t)=2^d\sum_{l=0}^d\left(\frac{1+e^{-4t}}{2}\right)^{d-l}\sum_{1\leqslant i_1<\cdots<i_l\leqslant d}\theta_{L(n_{i_1},\ldots,n_{i_l})}^\ast(t)
\end{equation*}
which is equivalent to
\begin{equation*}
(1+e^{-4t})^{-d}\theta_{T(2n_1,\ldots,2n_d)}(t)=\sum_{l=0}^d2^l(1+e^{-4t})^{-l}\sum_{1\leqslant i_1<\cdots<i_l\leqslant d}\theta_{L(n_{i_1},\ldots,n_{i_l})}^\ast(t).
\end{equation*}
By setting
\begin{equation*}
f(n_1,\ldots,n_d)=(1+e^{-4t})^{-d}\theta_{T(2n_1,\ldots,2n_d)}(t)\textnormal{ and }g(n_{i_1},\ldots,n_{i_l})=2^l(1+e^{-4t})^{-l}\theta_{L(n_{i_1},\ldots,n_{i_l})}^\ast(t)
\end{equation*}
in Lemma \ref{inversion}, it comes
\begin{equation*}
\theta_{L(n_{i_1},\ldots,n_{i_l})}^\ast(t)=2^{-l}\sum_{m=0}^l(-1-e^{-4t})^{l-m}\sum_{\substack{j_1<\cdots<j_m\\ \{j_q\}_{q=1}^m\subset\{i_q\}_{q=1}^l}}\theta_{T(2n_{j_1},\ldots,2n_{j_m})}(t).
\end{equation*}
From the above relation and relation (\ref{thetaLL*}), the theta function on the $d$-orthotope square lattice is expressed in terms of the theta function on the $d$-dimensional torus:
\begin{equation*}
\theta_{L(n_1,\ldots,n_d)}(t)=\sum_{l=0}^d\sum_{1\leqslant i_1<\cdots<i_l\leqslant d}\sum_{m=0}^l2^{-l}(-1-e^{-4t})^{l-m}\sum_{\substack{j_1<\cdots<j_m\\ \{j_q\}_{q=1}^m\subset\{i_q\}_{q=1}^l}}\theta_{T(2n_{j_1},\ldots,2n_{j_m})}(t).
\end{equation*}
Rewriting the double multi-index summation in one summation using (\ref{doublesumtoonesum}), we have
\begin{equation*}
\theta_{L(n_1,\ldots,n_d)}(t)=\sum_{l=0}^d\sum_{m=0}^lC^{l-m}_{d-m}2^{-l}(-1-e^{-4t})^{l-m}\sum_{1\leqslant i_1<\cdots<i_m\leqslant d}\theta_{T(2n_{i_1},\ldots,2n_{i_m})}(t)
\end{equation*}
Therefore the theta functions are related by
\begin{equation}
\label{thetaLT}
\theta_{L(n_1,\ldots,n_d)}(t)=\frac{1}{2^d}\sum_{m=0}^d(1-e^{-4t})^{d-m}\sum_{1\leqslant i_1<\cdots<i_m\leqslant d}\theta_{T(2n_{i_1},\ldots,2n_{i_m})}(t).
\end{equation}
\subsection{Preliminary calculation}
Let $\{\lambda_j^L\}_{j=0,\ldots,N-1}$ be the eigenvalues of the combinatorial Laplacian on the square lattice $L(n_1,\ldots,n_d)$. For small $t>0$, the theta function on the torus $\theta_{T(2n_{i_1},\ldots,2n_{i_m})}$ behaves as
\begin{equation}
\label{thetaT0}
\theta_{T(2n_{i_1},\ldots,2n_{i_m})}(t)-2^m\Big(\prod_{q=1}^mn_{i_q}\Big)e^{-2mt}I_0(2t)^m=O(t),\quad t\rightarrow0.
\end{equation}
We follow the method derived in \cite{chinta2010zeta}. From relation (\ref{thetaLT}) and the behaviour of the theta function at small $t>0$ (\ref{thetaT0}), we start by writing the theta function on the lattice $L(n_1,\ldots,n_d)$ as
\begin{align*}
\sum_{j\neq0}e^{-\lambda_j^Lt}&=\sum_{m=1}^d\left(\frac{1-e^{-4t}}{4}\right)^{d-m}e^{-2mt}I_0(2t)^mV_m^dn^m\\
&\ \ \ +\left[\theta_{L(n_1,\ldots,n_d)}(t)-\sum_{m=1}^d\left(\frac{1-e^{-4t}}{4}\right)^{d-m}e^{-2mt}I_0(2t)^mV_m^dn^m-1\right],
\end{align*}
to ensure the convergence of the integral of the Gauss transform that will appear below. By taking the Gauss transform of the above, that is, multiplying by $2se^{-s^2t}$ and then integrating with respect to $t$ from zero to infinity, we have
\begin{align}
\label{GaussT}
&\sum_{j\neq0}\frac{2s}{s^2+\lambda_j^L}=\sum_{m=1}^dV_m^dn^m2s\int_0^\infty e^{-s^2t}\left(\frac{1-e^{-4t}}{4}\right)^{d-m}e^{-2mt}I_0(2t)^mdt\nonumber\\
&+2s\left[\sum_{m=1}^d\int_0^\infty e^{-s^2t}(1-e^{-4t})^{d-m}\right.\nonumber\\
&\qquad\qquad\qquad\times\bigg(\frac{1}{2^d}\sum_{1\leqslant i_1<\cdots<i_m\leqslant d}\theta_{T(2n_{i_1},\ldots,2n_{i_m})}(t)-\frac{1}{4^{d-m}}V^d_mn^me^{-2mt}I_0(2t)^m-\frac{1}{2^d}C^m_d\bigg)dt\nonumber\\
&\left.\qquad+\int_0^\infty e^{-s^2t}((1-e^{-4t})^d-2^d+\sum_{m=1}^dC^m_d(1-e^{-4t})^{d-m})dt\right].
\end{align}
For $m\in\{1,\ldots,d\}$, define the functions $\mathcal{I}^d_m$ and $\mathcal{H}^d_{m,n}$ such that
\begin{equation}
\label{partialI}
\partial_s\mathcal{I}^d_m(s)=2s\int_0^\infty e^{-s^2t}\left(\frac{1-e^{-4t}}{4}\right)^{d-m}e^{-2mt}I_0(2t)^mdt
\end{equation}
and
\begin{align}
\label{partialH}
\partial_s\mathcal{H}^d_{m,n}(s)=2s&\int_0^\infty e^{-s^2t}(1-e^{-4t})^{d-m}\nonumber\\
&\times\bigg(\frac{1}{2^d}\sum_{1\leqslant i_1<\cdots<i_m\leqslant d}\theta_{T(2n_{i_1},\ldots,2n_{i_m})}(t)-\frac{1}{4^{d-m}}V^d_mn^me^{-2mt}I_0(2t)^m-\frac{1}{2^d}C^m_d\bigg)dt.
\end{align}
Equation (\ref{GaussT}) can then be written as
\begin{align}
\label{dsIO}
\sum_{j\neq0}\frac{2s}{s^2+\lambda_j^L}&=\sum_{m=1}^dV^d_mn^m\partial_s\mathcal{I}^d_m(s)+\sum_{m=1}^d\partial_s\mathcal{H}^d_{m,n}(s)+\frac{1}{2^d}2s\int_0^\infty e^{-s^2t}((2-e^{-4t})^d-2^d)dt.
\end{align}
By integrating over $s$ equations (\ref{partialI}) and (\ref{partialH}) we get
\begin{align*}
&\mathcal{I}^d_d(s)=\int_0^\infty(e^{-t}-e^{-(s^2+2d)t}I_0(2t)^d)\frac{dt}{t},\\
&\mathcal{H}^d_{d,n}(s)=-\frac{1}{2^d}\int_0^\infty(e^{-s^2t}(\theta_{T(2n_1,\ldots,2n_d)}(t)-V^d_d(2n)^de^{-2dt}I_0(2t)^d-1)+e^{-t})\frac{dt}{t}
\end{align*}
and for $m\neq d$,
\begin{align*}
&\mathcal{I}^d_m(s)=-\frac{1}{4^{d-m}}\int_0^\infty(1-e^{-4t})^{d-m}e^{-(s^2+2m)t}I_0(2t)^m\frac{dt}{t},\\
&\mathcal{H}^d_{m,n}(s)=-\int_0^\infty e^{-s^2t}(1-e^{-4t})^{d-m}\\
&\qquad\qquad\qquad\times\bigg(\frac{1}{2^d}\sum_{1\leqslant i_1<\cdots<i_m\leqslant d}\theta_{T(2n_{i_1,\ldots,2n_{i_m}})}(t)-\frac{1}{4^{d-m}}V^d_mn^me^{-2mt}I_0(2t)^m-\frac{1}{2^d}C^m_d\bigg)\frac{dt}{t}.
\end{align*}
By integrating equation (\ref{dsIO}) above we get
\begin{equation}
\sum_{j\neq0}\log(s^2+\lambda_j^L)=\sum_{m=1}^dV^d_mn^m\mathcal{I}^d_m(s)+\sum_{m=1}^d\mathcal{H}^d_{m,n}(s)+\sum_{k=1}^dC^k_d(-2)^{-k}\log(s^2+4k)+\textnormal{constant}.
\label{log(s2+lj)}
\end{equation}
The asymptotic behaviour of the functions $\mathcal{I}^d_m$ and $\mathcal{H}^d_{m,n}$ as $s\rightarrow\infty$ determine the constant of integration. We have
\begin{equation*}
\mathcal{I}^d_d(s)=2\log{s}+o(1)\textnormal{ and }\mathcal{I}^d_m(s)=o(1)\textnormal{ for }m\neq d\textnormal{ as }s\rightarrow\infty
\end{equation*}
and
\begin{equation*}
\mathcal{H}^d_{d,n}(s)=-\frac{1}{2^{d-1}}\log{s}+o(1)\textnormal{ and }\mathcal{H}^d_{m,n}(s)=o(1)\textnormal{ for }m\neq d\textnormal{ as }s\rightarrow\infty
\end{equation*}
and
\begin{equation*}
\sum_{j\neq0}\log(s^2+\lambda_j^L)=2(N-1)\log{s}+o(1)\textnormal{ as }s\rightarrow\infty.
\end{equation*}
Therefore, equation (\ref{log(s2+lj)}) as $s\rightarrow\infty$ yields
\begin{equation*}
2(N-1)\log{s}+o(1)=2N\log{s}-\frac{1}{2^{d-1}}\log{s}+\frac{1}{2^{d-1}}(1-2^d)\log{s}+\textnormal{constant}+o(1)
\end{equation*}
so that the constant is zero. Evaluating equation (\ref{log(s2+lj)}) in $s=0$ gives the logarithm of the product of the non-zero eigenvalues of the Laplacian on the $d$-dimensional square lattice $L(n_1,\ldots,n_d)$:
\begin{equation}
\log\Big(\prod_{j\neq0}\lambda_j^L\Big)=c_dV^d_dn^d+\sum_{m=1}^{d-1}V^d_m\mathcal{I}^d_m(0)n^m+\sum_{m=1}^d\mathcal{H}^d_{m,n}(0)+\sum_{k=1}^dC^k_d(-2)^{-k}\log(4k)
\label{logdets=0}
\end{equation}
where
\begin{equation*}
c_d=\int_0^\infty(e^{-t}-e^{-2dt}I_0(2t)^d)\frac{dt}{t}.
\end{equation*}
\subsection{Asymptotic expansion}
By expanding $(1-e^{-4t})^{d-m}=\sum_{k=0}^{d-m}C^k_{d-m}(-1)^ke^{-4kt}$ in $\mathcal{H}^d_{m,n}(0)$, it can be rewritten as
\begin{align*}
\mathcal{H}^d_{m,n}(0)&=-\frac{1}{2^d}\sum_{k=0}^{d-m}C^k_{d-m}(-1)^k\\
&\times\sum_{1\leqslant i_1<\cdots<i_m\leqslant d}\int_0^\infty\big(e^{-4kt}(\theta_{T(2n_{i_1},\ldots,2n_{i_m})}(t)-\prod_{q=1}^m(2\alpha_{i_q}ne^{-2t}I_0(2t))-1)+e^{-t}\big)\frac{dt}{t},
\end{align*}
for all $m=1,\ldots,d$, where the $e^{-t}$ term is added to make the integral converge. It can be added since $\sum_{k=0}^{d-m}C^k_{d-m}(-1)^k=0$ for $m=1,\ldots,d-1$. By splitting the sum over $k$ we have
\begin{align*}
&\sum_{m=1}^d\mathcal{H}^d_{m,n}(0)=\\
&-\frac{1}{2^d}\sum_{m=1}^d\sum_{1\leqslant i_1<\cdots<i_m\leqslant d}\int_0^\infty\big(\theta_{T(2n_{i_1},\ldots,2n_{i_m})}(t)-\prod_{q=1}^m(2\alpha_{i_q}ne^{-2t}I_0(2t))-1+e^{-t}\big)\frac{dt}{t}\\
&-\frac{1}{2^d}\sum_{m=1}^d\sum_{k=1}^{d-m}\sum_{1\leqslant i_1<\cdots<i_m\leqslant d}\int_0^\infty\big(e^{-4kt}(\theta_{T(2n_{i_1},\ldots,2n_{i_m})}(t)-\prod_{q=1}^m(2\alpha_{i_q}ne^{-2t}I_0(2t))-1)+e^{-t}\big)\frac{dt}{t}.
\end{align*}
From \cite[Theorem $5.8$]{chinta2010zeta}, the asymptotic behaviour as $n\rightarrow\infty$ of the $k=0$ term is given by
\begin{align*}
&-\frac{1}{2^d}\sum_{m=1}^d\sum_{1\leqslant i_1<\cdots<i_m\leqslant d}\int_0^\infty(\theta_{T(2n_{i_1},\ldots,2n_{i_m})}(t)-\prod_{q=1}^m(2\alpha_{i_q}ne^{-2t}I_0(2t))-1+e^{-t})\frac{dt}{t}\\
&=\frac{1}{2^d}\sum_{m=1}^d\sum_{1\leqslant i_1<\cdots<i_m\leqslant d}(2\log{n}-\zeta'_{\mathbb{R}^m/\textnormal{diag}(2\alpha_{i_1},\ldots,2\alpha_{i_m})\mathbb{Z}^m}(0))+o(1).
\end{align*}
After a change of variable $t\rightarrow n^2t$, the sum over the non-zero $k$'s can be splitted as
\begin{align}
&-\frac{1}{2^d}\sum_{k=1}^{d-1}\sum_{m=1}^{d-k}\sum_{1\leqslant i_1<\cdots<i_m\leqslant d}C^k_{d-m}(-1)^k\nonumber\\
&\times\left[\int_0^1e^{-4kn^2t}(\theta_{T(2n_{i_1},\ldots,2n_{i_m})}(n^2t)-\prod_{q=1}^m(2\alpha_{i_q}ne^{-2n^2t}I_0(2n^2t)))\frac{dt}{t}+\int_0^1(e^{-n^2t}-e^{-4kn^2t})\frac{dt}{t}\right.\nonumber\\
&\quad\quad+\int_1^\infty e^{-4kn^2t}(\theta_{T(2n_{i_1},\ldots,2n_{i_m})}(n^2t)-1)\frac{dt}{t}\nonumber\\
&\quad\quad+\left.\int_1^\infty(e^{-n^2t}-e^{-4kn^2t}\prod_{q=1}^m(2\alpha_{i_q}ne^{-2n^2t}I_0(2n^2t)))\frac{dt}{t}\right].
\label{Hk>0}
\end{align}
From the propositions in \cite[section $5$]{chinta2010zeta}, the first, third and fourth integrals tend to zero as $n\rightarrow\infty$. The second integral tends to
\begin{equation*}
\int_0^\infty(e^{-t}-e^{-4kt})\frac{dt}{t}=\log(4k).
\end{equation*}
The limit as $n\rightarrow\infty$ of (\ref{Hk>0}) is then
\begin{equation*}
-\frac{1}{2^d}\sum_{k=1}^{d-1}\sum_{m=1}^{d-k}C^m_dC^k_{d-m}(-1)^k\log(4k)=-\frac{1}{2^d}\sum_{k=1}^{d-1}C^k_d(2^{d-k}-1)(-1)^k\log(4k).
\end{equation*}
Therefore, the $\mathcal{H}^d_{m,n}(0)$ term together with the constant term of equation (\ref{logdets=0}) behave as
\begin{equation*}
\left(2-\frac{1}{2^{d-1}}\right)\log{n}-\frac{1}{2^d}\sum_{m=1}^d\sum_{1\leqslant i_1<\cdots<i_m\leqslant d}\zeta'_{\mathbb{R}^m/\textnormal{diag}(2\alpha_{i_1},\ldots,2\alpha_{i_m})\mathbb{Z}^m}(0)+\frac{1}{2^d}\sum_{k=1}^dC^k_d(-1)^k\log(4k)
\end{equation*}
as $n\rightarrow\infty$.\\
We will now express the derivative of the spectral zeta function on the $m$-dimensional real torus $\mathbb{R}^m/\textnormal{diag}(2\alpha_{i_1},\ldots,2\alpha_{i_m})\mathbb{Z}^m$, that is $\zeta'_{\mathbb{R}^m/\textnormal{diag}(2\alpha_{i_1},\ldots,2\alpha_{i_m})\mathbb{Z}^m}$, in terms of the derivative of the spectral zeta function on the $m$-dimensional orthotope of size $\alpha_{i_1}\times\cdots\times\alpha_{i_m}$. The eigenvalues of the Laplace-Beltrami operator with Dirichlet boundary conditions on the $m$-dimensional orthotope of size $\alpha_{i_1}\times\cdots\times\alpha_{i_m}$ are given by
\begin{equation*}
\lambda_k=\pi^2\sum_{q=1}^m\left(\frac{k_{i_q}}{\alpha_{i_q}}\right)^2\textnormal{ with }k=(k_{i_1},\ldots,k_{i_m})\in(\mathbb{N}^\ast)^m.
\end{equation*}
So that the spectral zeta function on the $m$-dimensional orthotope of size $\alpha_{i_1}\times\cdots\times\alpha_{i_m}$ with Dirichlet boundary conditions, denoted by $\zeta_{\alpha_{i_1}\times\cdots\times\alpha_{i_m}}$, is given by
\begin{equation*}
\zeta_{\alpha_{i_1}\times\cdots\times\alpha_{i_m}}(s)=\frac{1}{\pi^{2s}}\sum_{k_{i_1},\ldots,k_{i_m}\geqslant1}\left(\sum_{q=1}^m\left(\frac{k_{i_q}}{\alpha_{i_q}}\right)^2\right)^{-s}.
\end{equation*}
The spectral zeta function on the real torus $\mathbb{R}^m/\textnormal{diag}(2\alpha_{i_1},\ldots,2\alpha_{i_m})\mathbb{Z}^m$ is given by
\begin{equation*}
\zeta_{\mathbb{R}^m/\textnormal{diag}(2\alpha_{i_1},\ldots,2\alpha_{i_m})\mathbb{Z}^m}(s)=\frac{1}{(2\pi)^{2s}}\sum_{(k_{i_1},\ldots,k_{i_m})\in\mathbb{Z}^m\backslash\{0\}}\left(\sum_{q=1}^m\left(\frac{k_{i_q}}{2\alpha_{i_q}}\right)^2\right)^{-s}.
\end{equation*}
The spectral zeta functions are then related by
\begin{equation*}
\zeta_{\mathbb{R}^m/\textnormal{diag}(2\alpha_{i_1},\ldots,2\alpha_{i_m})\mathbb{Z}^m}(s)=\sum_{l=1}^m2^l\sum_{\substack{j_1<\cdots<j_l\\ \{j_q\}_{q=1}^l\subset\{i_q\}_{q=1}^m}}\zeta_{\alpha_{j_1}\times\cdots\times\alpha_{j_l}}(s).
\end{equation*}
Summing the above over all $i_q$, $q=1,\ldots,m$ and over all $m$, $m=1,\ldots,d$, gives
\begin{align*}
&\frac{1}{2^d}\sum_{m=1}^d\sum_{1\leqslant i_1<\cdots<i_m\leqslant d}\zeta_{\mathbb{R}^m/\textnormal{diag}(2\alpha_{i_1},\ldots,2\alpha_{i_m})\mathbb{Z}^m}(s)\\
&=\frac{1}{2^d}\sum_{m=1}^d\sum_{1\leqslant i_1<\cdots<i_m\leqslant d}\sum_{l=1}^m2^l\sum_{\substack{j_1<\cdots<j_l\\ \{j_q\}_{q=1}^l\subset\{i_q\}_{q=1}^m}}\zeta_{\alpha_{j_1}\times\cdots\times\alpha_{j_l}}(s)\\
&=\frac{1}{2^d}\sum_{m=1}^d\sum_{l=1}^mC^{m-l}_{d-l}2^l\sum_{1\leqslant i_1<\cdots<i_l\leqslant d}\zeta_{\alpha_{i_1}\times\cdots\times\alpha_{i_l}}(s)\\
&=\sum_{m=1}^d\sum_{1\leqslant i_1<\cdots<i_m\leqslant d}\zeta_{\alpha_{i_1}\times\cdots\times\alpha_{i_m}}(s)
\end{align*}
where in the last equality we exchanged the sums over $m$ and $l$ and used the fact that \linebreak$\sum_{m=l}^dC^{m-l}_{d-l}=2^{d-l}$. By expressing the derivative of the spectral zeta function evaluated in zero in terms of the regularized determinant of the Laplace-Beltrami operator on $m$-dimensional orthotopes, $m=1,\ldots,d$, with Dirichlet boundary conditions, we have
\begin{equation*}
\frac{1}{2^d}\sum_{m=1}^d\sum_{1\leqslant i_1<\cdots<i_m\leqslant d}\zeta'_{\mathbb{R}^m/\textnormal{diag}(2\alpha_{i_1},\ldots,2\alpha_{i_m})\mathbb{Z}^m}(0)=-\sum_{m=1}^d\sum_{1\leqslant i_1<\cdots<i_m\leqslant d}\log\textnormal{det}^\ast\Delta_{\alpha_{i_1}\times\cdots\times\alpha_{i_m}}.
\end{equation*}
Putting everything together gives the asymptotic behaviour of the determinant of the Laplacian on the $d$-dimensional square lattice $L(n_1,\ldots,n_d)$
\begin{align*}
\log\textnormal{det}^\ast\Delta_{L(\alpha_1n,\ldots,\alpha_dn)}&=c_dV^d_dn^d-\sum_{m=1}^{d-1}\frac{1}{4^{d-m}}\left(\int_0^\infty(1-e^{-4t})^{d-m}e^{-2mt}I_0(2t)^m\frac{dt}{t}\right)V^d_mn^m\\
&\ \ \ +(2-2^{1-d})\log{n}+\sum_{m=1}^d\sum_{1\leqslant i_1<\cdots<i_m\leqslant d}\log\textnormal{det}^\ast\Delta_{\alpha_{i_1}\times\cdots\times\alpha_{i_m}}\\
&\ \ \ +\frac{1}{2^d}\sum_{m=1}^dC^m_d(-1)^m\log(4m)+o(1)
\end{align*}
as $n\rightarrow\infty$. Notice that in the bulk limit the lead term is the same as in the case of the torus but lower order terms are deducted; for each $m$-dimensional face, $m=1,\ldots,d-1$, a term proportional to $V_m^dn^m$ is deducted. This can be explained by observing that in the torus case we have a periodic lattice while for the $d$-dimensional hypercubic lattice the periodicity is substituted by free boundary conditions. Spectral determinants of the limiting $d$-dimensional orthotope and each of its $m$-dimensional faces, $m=1,\ldots,d-1$, also appear. Moreover the last constant term is new in the development. The terms in $V_m^dn^m$ appearing from the boundary effect can be written in the following way:
\begin{align*}
-\int_0^\infty(1-e^{-4t})^{d-m}e^{-2mt}I_0(2t)^m\frac{dt}{t}&=\sum_{k=0}^{d-m}C^k_{d-m}(-1)^k\int_0^\infty(e^{-t}-e^{-(2k+m)t}I_0(t)^m)\frac{dt}{t}.
\end{align*}
These integrals are denoted by $J_m(2k+m)$ in \cite{MR2999720} by Glasser and are related to the Mahler measure of the hypercubic polynomial $P(x_1,\ldots,x_m)=4k+2m+\sum_{j=1}^m(x_j+x_j^{-1})$ by the relation $m(P)=\log{2}+J_m(2k+m)$. In \cite{MR2999720}, Glasser calculates these integrals in terms of hypergeometric functions for $m=2$ and $3$. For $m=1$, one explicity has
\begin{equation*}
\frac{1}{4^{d-1}}\int_0^\infty(1-e^{-4t})^{d-1}e^{-2t}I_0(2t)\frac{dt}{t}=\frac{1}{4^{d-1}}\sum_{k=1}^{d-1}C^k_{d-1}(-1)^k\log(2k+1+2\sqrt{k^2+k})
\end{equation*}
(see \cite[Proposition $2.4$]{louis2015asymptotics}).
\begin{example}
Consider the two dimensional rectangular grid $\alpha_1n\times\alpha_2n$. The volume of the limiting rectangle of size $\alpha_1\times\alpha_2$ is $V^2_2=\alpha_1\alpha_2$ with perimeter $V^2_1=2\alpha_1+2\alpha_2$. From Theorem \ref{ThO} it comes
\begin{align*}
\log\textnormal{det}^\ast\Delta_{L(\alpha_1n,\alpha_2n)}&=\frac{4G}{\pi}V^2_2n^2-\frac{1}{2}\log(1+\sqrt{2})V^2_1n+\frac{3}{2}\log{n}+\log\textnormal{det}^\ast\Delta_{\alpha_1\times\alpha_2}\\
&\ \ \ +\log\textnormal{det}^\ast\Delta_{\alpha_1}+\log\textnormal{det}^\ast\Delta_{\alpha_2}-\frac{1}{4}\log{2}+o(1)
\end{align*}
as $n\rightarrow\infty$, which is equivalent to the formula derived in \cite[section $4.2$]{MR952941}.
\end{example}
\subsection{\texorpdfstring{Asymptotic number of rooted \boldmath$2$- and \boldmath$3$-spanning forests}{Asymptotic number of 2 and 3-spanning forests}}
In this section, we derive asymptotics for $2$ and $3$ component rooted spanning forests in $d$-orthotope square lattices. Let $N^d_k$ denote the number of rooted $k$-spanning forests on $L(n_1,\ldots,n_d)$ which is given by the $k$-th power in $s^2$ of the characteristic polynomial:
\begin{equation*}
\prod_{j=0}^{\prod_{i=1}^dn_i-1}\left(\lambda_j^L+\frac{s^2}{n^2}\right).
\end{equation*}
Following \cite{MR952941}, by expanding the above in powers of $(s/n)^2$, one finds that $N^d_2$ and $N^d_3$ are related to $N^d_1$ by
\begin{equation*}
\frac{N^d_2}{N^d_1}=\sum_{j\neq0}\frac{1}{\lambda_j^L}\quad\textnormal{and}\quad\frac{N^d_3}{N^d_1}=\frac{1}{2}\left(\left(\frac{N^d_2}{N^d_1}\right)^2-\sum_{j\neq0}\frac{1}{(\lambda_j^L)^2}\right).
\end{equation*}
The number or rooted spanning trees $N^d_1$ is related to the number of unrooted spanning trees $\tau(L(n_1,\ldots,n_d))$ by the relation
\begin{equation*}
N^d_1=\Big(\prod_{i=1}^dn_i\Big)\tau(L(n_1,\ldots,n_d)).
\end{equation*}
Recall that from equation (\ref{log(s2+lj)}), we have that
\begin{equation*}
\sum_{j\neq0}\log((s/n)^2+\lambda_j^L)=\sum_{m=1}^dV^d_mn^m\mathcal{I}^d_m(s/n)+\sum_{m=1}^d\mathcal{H}^d_{m,n}(s/n)+\sum_{k=1}^dC^k_d(-2)^{-k}\log((s/n)^2+4k)
\end{equation*}
where
\begin{equation*}
\mathcal{I}^d_d(s/n)=\int_0^\infty(e^{-t}-e^{-((s/n)^2+2d)t}I_0(2t)^d)\frac{dt}{t}.
\end{equation*}
We have
\begin{equation*}
\lim_{n\rightarrow\infty}\mathcal{I}^d_d(s/n)=c_d.
\end{equation*}
For $d\geqslant3$,
\begin{equation*}
\lim_{n\rightarrow\infty}n^2(\mathcal{I}^d_d(s/n)-c_d)=\lim_{n\rightarrow\infty}n^2\int_0^\infty(1-e^{-(s/n)^2t})e^{-2dt}I_0(2t)^d\frac{dt}{t}=\frac{s^2}{2}W_d
\end{equation*}
where $W_d$ is the so-called Watson integral for the $d$-dimensional hypercubic lattice
\begin{equation*}
W_d=\int_0^\infty e^{-dt}I_0(t)^ddt.
\end{equation*}
In \cite{MR1862771}, Joyce and Zucker introduced the generalised lattice Green function defined by
\begin{equation*}
G_d(\mathbf{n};k,w)=\frac{1}{\Gamma(k)}\int_0^\infty t^{k-1}e^{-wt}\prod_{i=1}^dI_{n_i}(t)dt
\end{equation*}
where $\mathbf{n}=\{n_1,\ldots,n_d\}$ is a set of non-negative integers, $w\geqslant d$, $k>0$ and $\Gamma$ is the gamma function. Here the lattice Green function will only appear with $\mathbf{n}=0$, hence we denote it shortly by $G_d(k,w)$. In \cite{MR1862771}, numerical evaluations of the integrals $c_d$ and $W_d$ are computed and also in \cite{chinta2010zeta} for $c_d$ and in \cite{MR1965463} for $W_d$ and $G_d(1,w)$.\\
For $d\geqslant5$,
\begin{align*}
\lim_{n\rightarrow\infty}(n^4(\mathcal{I}^d_d(s/n)-c_d)-n^2s^2W_d)&=\lim_{n\rightarrow\infty}n^4\int_0^\infty\left(1-e^{-(s/n)^2t}-\frac{s^2}{n^2}t\right)e^{-2dt}I_0(2t)^d\frac{dt}{t}\\
&\ \ \ =-\frac{s^4}{8}G_d(2,d).
\end{align*}
Continuing in this way, we arrive at the following expansion for $\mathcal{I}^d_d(s/n)$ as $n\rightarrow\infty$
\begin{equation}
n^d\mathcal{I}^d_d(s/n)=c_dn^d+\sum_{k=1}^{\lfloor(d-1)/2\rfloor}(-1)^{k+1}n^{d-2k}\frac{s^{2k}}{k2^k}G_d(k,d)+o(n).
\label{Id(s/n)}
\end{equation}
Recall that for $m\in\{1,\ldots,d-1\}$,
\begin{equation*}
\mathcal{I}^d_m(s/n)=-\frac{1}{4^{d-m}}\int_0^\infty(1-e^{-4t})^{d-m}e^{-((s/n)^2+2m)t}I_0(2t)^m\frac{dt}{t}.
\end{equation*}
We have
\begin{equation*}
\lim_{n\rightarrow\infty}\mathcal{I}^d_m(s/n)=\mathcal{I}^d_m(0).
\end{equation*}
Similarly as for $m=d$, we obtain as $n\rightarrow\infty$
\begin{align}
n^m\mathcal{I}^d_m(s/n)&=\mathcal{I}^d_m(0)n^m\nonumber\\
&+\frac{1}{4^{d-m}}\sum_{k=1}^{\lfloor(m-1)/2\rfloor}(-1)^kn^{m-2k}\frac{s^{2k}}{k!2^k}\int_0^\infty(1-e^{-2t})^{d-m}e^{-mt}I_0(t)^mt^{k-1}dt+o(n).
\label{Im(s/n)}
\end{align}
The above integral can be expressed in terms of the generalised lattice Green function:
\begin{equation*}
\frac{1}{(k-1)!}\int_0^\infty(1-e^{-2t})^{d-m}e^{-mt}I_0(t)^mt^{k-1}dt=\sum_{l=0}^{d-m}C^l_{d-m}(-1)^lG_m(k,m+2l).
\end{equation*}
Putting equations (\ref{Id(s/n)}) and (\ref{Im(s/n)}) together gives the expansion for $d\geqslant3$
\begin{align*}
\sum_{j\neq0}\log(&(s/n)^2+\lambda_j^L)=V_d^dc_dn^d+\sum_{m=1}^{d-1}V_m^d\mathcal{I}^d_m(0)n^m+\sum_{k=1}^{\lfloor(d-1)/2\rfloor}\frac{s^{2k}}{k2^k}\bigg((-1)^{k+1}V^d_dG_d(k,d)n^{d-2k}\\
&+\delta_{d\geqslant4}\frac{(-1)^k}{(k-1)!}\sum_{m=2k+1}^{d-1}V^d_mn^{m-2k}\frac{1}{4^{d-m}}\int_0^\infty(1-e^{-2t})^{d-m}e^{-mt}I_0(t)^mt^{k-1}dt\bigg)+o(n)
\end{align*}
 where $\delta_{d\geqslant d_0}=1$ if $d\geqslant d_0$ and $0$ otherwise.\\
For $d=3$,
\begin{align*}
\sum_{j\neq0}\log((s/n)^2&+\lambda_j^L)=V_3^3c_3n^d+V_1^3\mathcal{I}^3_1(0)n+V_2^3\mathcal{I}^3_2(0)n^2+\frac{s^2}{2}V^3_3W_3n+o(n)
\end{align*}
with the special values ($W_3$ is given in \cite{MR1862771})
\begin{equation*}
\mathcal{I}^3_1(0)=\frac{1}{16}\log((17+2\sqrt{2})(5-2\sqrt{6}))\quad\textnormal{and}\quad W_3=\frac{1}{96\pi^3}(\sqrt{3}-1)\left(\Gamma(1/24)\Gamma(11/24)\right)^2.
\end{equation*}
For $d\geqslant4$,
\begin{align*}
&\sum_{j\neq0}\log((s/n)^2+\lambda_j^L)=V^d_dc_dn^d+\sum_{m=1}^{d-1}V^d_m\mathcal{I}^d_m(0)n^m\\
&+\frac{s^2}{2}\left(V^d_dW_dn^{d-2}-\sum_{m=3}^{d-1}V^d_mn^{m-2}\frac{1}{4^{d-m}}\int_0^\infty(1-e^{-2t})^{d-m}e^{-mt}I_0(t)^mdt\right)\\
&+\frac{s^4}{8}\left(-\delta_{d\geqslant5}V^d_dG_d(2,d)n^{d-4}+\delta_{d\geqslant6}\sum_{m=5}^{d-1}V_m^dn^{m-4}\frac{1}{4^{d-m}}\int_0^\infty t(1-e^{-2t})^{d-m}e^{-mt}I_0(t)^mdt\right)\\
&+(\textnormal{terms in }s^k\textnormal{ with }k\geqslant3)+o(n)
\end{align*}
as $n\rightarrow\infty$.\\
On the other hand, the formal expansion of $\sum_{j\neq0}\log((s/n)^2+\lambda_j^L)$ gives
\begin{equation*}
\sum_{j\neq0}\log((s/n)^2+\lambda_j^L)=\log\bigg(\prod_{j\neq0}\lambda_j^L\bigg)+\sum_{p\geqslant1}\frac{(-1)^{p-1}}{p}\left(\frac{s}{n}\right)^{2p}\sum_{j\neq0}\frac{1}{(\lambda_j^L)^p}.
\end{equation*}
By identification of the terms in $s^2$, we find the asymptotic number of rooted $2$-spanning forests for $d=3$, as $n\rightarrow\infty$
\begin{equation*}
N^3_2=\left(\frac{V^3_3}{2}W_3n^3+o(n^3)\right)N^3_1
\end{equation*}
and for $d\geqslant4$, we have
\begin{equation*}
\sum_{j\neq0}\frac{1}{\lambda_j^L}=\frac{V^d_d}{2}W_dn^d-\sum_{m=3}^{d-1}\frac{V^d_m}{2}n^m\frac{1}{4^{d-m}}\int_0^\infty(1-e^{-2t})^{d-m}e^{-mt}I_0(t)^mdt+o(n^3)
\end{equation*}
so that
\begin{equation*}
N^d_2=\left(\frac{V^d_d}{2}W_dn^d-\sum_{m=3}^{d-1}\frac{V^d_m}{2}n^m\frac{1}{4^{d-m}}\int_0^\infty(1-e^{-2t})^{d-m}e^{-mt}I_0(t)^mdt+o(n^3)\right)N^d_1
\end{equation*}
where $N^d_1$ is asymptotically given by Theorem \ref{ThO}. By identification of the terms in $s^4$, we find that as $n\rightarrow\infty$
\begin{equation*}
\sum_{j\neq0}\frac{1}{(\lambda_j^L)^2}=O(n^d)
\end{equation*}
so that the asymptotic number of rooted $3$-spanning forests for $d=4$ is given by
\begin{equation*}
N^4_3=\left(\frac{(V^4_4W_4)^2}{8}n^8-\frac{V^4_3V^4_4}{8}W_4n^7\int_0^\infty(1-e^{-2t})e^{-3t}I_0(t)^3dt+o(n^7)\right)N^4_1,\textnormal{ as }n\rightarrow\infty.
\end{equation*}
\begin{remark}
It would be interesting to find the next terms in the development. For the $2$-dimensional case, we would need to find the asymptotic development as $n\rightarrow\infty$ of the following integral
\begin{equation*}
\int_0^\infty n^2(1-e^{-(s/n)^2t})e^{-4t}I_0(2t)^2\frac{dt}{t}.
\end{equation*}
In \cite{MR952941}, the authors computed the asymptotic development of $\sum_{j\neq0}\log((s/n)^2+\lambda_j)$ in the case of the torus with other techniques. To generalise their result to higher dimensions with our techniques, for example in the $3$-dimensional case, we would need to find the asymptotic development of
\begin{equation*}
\int_0^\infty(n^3(1-e^{-(s/n)^2t})-ns^2t)e^{-6t}I_0(2t)^3\frac{dt}{t}
\end{equation*}
as $n\rightarrow\infty$. This would enable us to derive asymptotics for the number of rooted $k$-spanning forests with $k\geqslant4$.
\end{remark}
\subsection{Spanning trees in two-dimensional square lattices}
In the two-dimensional case, one can derive an exact relation between the number of spanning trees on the rectangular square lattice $n_1\times n_2$ and the one on the torus of size $2n_1\times2n_2$. The product of the non-zero eigenvalues on $T(2n_1,2n_2)$ is given by
\begin{equation*}
\textnormal{det}^\ast\Delta_{T(2n_1,2n_2)}=\prod_{\substack{k_1=0\\ \quad\quad\mathclap{(k_1,k_2)\neq0}}}^{2n_1-1}\prod_{k_2=0}^{2n_2-1}(4-2\cos(\pi k_1/n_1)-2\cos(\pi k_2/n_2)).
\end{equation*}
The product over $k_1$, $k_2$ is a disjoint union of products over four squares of size $n_1\times n_2$. We split this product as a product over the $0$, $1$ and $2$ dimensional faces of the squares. It comes
\begin{align}
\textnormal{det}^\ast\Delta_{T(2n_1,2n_2)}&=4^28\prod_{k_1=1}^{n_1-1}(2-2\cos(\pi k_1/n_1))^2\prod_{k_2=1}^{n_2-1}(2-2\cos(\pi k_2/n_2))^2\nonumber\\
&\ \ \ \times\prod_{k_1=1}^{n_1-1}(6-2\cos(\pi k_1/n_1))^2\prod_{k_2=1}^{n_2-1}(6-2\cos(\pi k_2/n_2))^2\nonumber\\
&\ \ \ \times\prod_{k_1=1}^{n_1-1}\prod_{k_2=1}^{n_2-1}(4-2\cos(\pi k_1/n_1)-2\cos(\pi k_2/n_2))^4.
\label{prodT(2n1,2n2)}
\end{align}
On the other hand, the product of the non-zero eigenvalues on the square lattice $L(n_1,n_2)$ is given by
\begin{equation*}
\textnormal{det}^\ast\Delta_{L(n_1,n_2)}=\prod_{\substack{k_1=0\\ \quad\quad\mathclap{(k_1,k_2)\neq0}}}^{n_1-1}\prod_{k_2=0}^{n_2-1}(4-2\cos(\pi k_1/n_1)-2\cos(\pi k_2/n_2)).
\end{equation*}
By splitting the above product as a product when $k_1=0$, then $k_2=0$, then $1\leqslant k_1\leqslant n_1-1$, $1\leqslant k_2\leqslant n_2-1$, we get
\begin{align}
\textnormal{det}^\ast\Delta_{L(n_1,n_2)}&=\prod_{k_1=1}^{n_1-1}(2-2\cos(\pi k_1/n_1))\prod_{k_2=1}^{n_2-1}(2-2\cos(\pi k_2/n_2))\nonumber\\
&\ \ \ \times\prod_{k_1=1}^{n_1-1}\prod_{k_2=1}^{n_2-1}(4-2\cos(\pi k_1/n_1)-2\cos(\pi k_2/n_2)).
\label{prodL(n1,n2)}
\end{align}
Using the matrix tree theorem and putting the following identities coming from relations for Chebyshev polynomials of the second kind
\begin{equation*}
\prod_{k=1}^{n-1}(2-2\cos(\pi k/n))=n\textnormal{ and }\prod_{k=1}^{n-1}(6-2\cos(\pi k/n))=\frac{(3+2\sqrt{2})^n-(3-2\sqrt{2})^n}{4\sqrt{2}},
\end{equation*}
in equations (\ref{prodT(2n1,2n2)}) and (\ref{prodL(n1,n2)}), it follows that
\begin{equation*}
\tau(L(n_1,n_2))=\frac{2^{5/4}\tau(T(2n_1,2n_2))^{1/4}}{(n_1n_2)^{1/4}((3+2\sqrt{2})^{n_1}-(3-2\sqrt{2})^{n_1})^{1/2}((3+2\sqrt{2})^{n_2}-(3-2\sqrt{2})^{n_2})^{1/2}}.
\end{equation*}
\begin{remark}
It would be interesting to see if one could generalise the above relation to higher dimensions. It could not be done in the same way as it is done above. More precisely, when splitting the product in $3$ dimensions as a product over $0$, $1$, $2$ and $3$ dimensional faces, one would need for example to evaluate the following product
\begin{equation*}
\prod_{k_1=1}^{n_1-1}\prod_{k_2=1}^{n_2-1}(8-2\cos(\pi k_1/n_1)-2\cos(\pi k_2/n_2))
\end{equation*}
appearing for the $2$-dimensional face defined by $k_3=n_3$ and $k_1=1,\ldots,n_1-1$, $k_2=1,\ldots,n_2-1$.
\end{remark}
\section{Asymptotic number of spanning trees in the quartered Aztec diamond}
\subsection{A relation between the number of spanning trees on the quartered Aztec diamond and on the square lattice}
In \cite{MR2393253,MR1756162}, the authors showed that the number of spanning trees in the quartered Aztec diamond of side length $n$ is given by
\begin{equation*}
\tau(\textnormal{QAD}_n)=\prod_{0<k_1<k_2<n}(4-2\cos(\pi k_1/n)-2\cos(\pi k_2/n)).
\end{equation*}
The product of the non-zero eigenvalues on the square grid of side $n$ is given by
\begin{equation*}
\textnormal{det}^\ast\Delta_{L(n,n)}=\prod_{\substack{k_1=0\\ \quad\quad\mathclap{(k_1,k_2)\neq0}}}^{n-1}\prod_{k_2=0}^{n-1}(4-2\cos(\pi k_1/n)-2\cos(\pi k_2/n)).
\end{equation*}
By splitting this product as a product when $k_1=0$, then $k_2=0$, then $k_1=k_2$, $k_1=1,\ldots,n-1$, then $k_1<k_2$ and $k_2<k_1$, we have
\begin{align*}
\textnormal{det}^\ast\Delta_{L(n,n)}&=\prod_{k=1}^{n-1}(2-2\cos(\pi k/n))^2\prod_{k=1}^{n-1}(4-4\cos(\pi k/n))\\
&\ \ \ \times\prod_{1\leqslant k_1<k_2\leqslant n-1}(4-2\cos(\pi k_1/n)-2\cos(\pi k_2/n))^2.
\end{align*}
From the matrix tree theorem, it follows that
\begin{equation}
\label{tau(QAD)}
\tau(\textnormal{QAD}_n)=\frac{\tau(L(n,n))^{1/2}}{\sqrt{n}2^{(n-1)/2}}.
\end{equation}
\subsection{Asymptotic expansion}
From (\ref{tau(QAD)}), we have
\begin{equation}
\label{log(tau(QAD))}
\log(\tau(QAD_n))=\frac{1}{2}\log\textnormal{det}^\ast\Delta_{L(n,n)}-\frac{n}{2}\log{2}-\frac{3}{2}\log{n}+\frac{1}{2}\log{2}
\end{equation}
where the asymptotic behaviour of $\log\textnormal{det}^\ast\Delta_{L(n,n)}$ is given by
\begin{equation}
\label{L(n,n)}
\log\textnormal{det}^\ast\Delta_{L(n,n)}=\frac{4G}{\pi}n^2-2\log(1+\sqrt{2})n+\frac{3}{2}\log{n}-\zeta'_{1\times1}(0)-2\zeta'_{1}(0)-\frac{1}{4}\log{2}+o(1)
\end{equation}
as $n\rightarrow\infty$.\\
Consider the right-angled isoscele triangle with the sides of same length equal to $1$. The eigenvalues of the Laplace-Beltrami operator with Dirichlet boundary conditions are given by
\begin{equation*}
\lambda_k=\pi^2(k_1^2+k_2^2),\quad\textnormal{with }k=(k_1,k_2)\in(\mathbb{N}^\ast)^2\textnormal{ and }k_1>k_2.
\end{equation*}
The associated spectral zeta function with Dirichlet boundary conditions, denoted by $\zeta_\Delta$, is then given by
\begin{equation*}
\zeta_{\Delta}(s)=\frac{1}{\pi^{2s}}\sum_{1\leqslant k_2<k_1}\frac{1}{(k_1^2+k_2^2)^s}.
\end{equation*}
The spectral zeta function on the unit square with Dirichlet boundary conditions is given by
\begin{equation*}
\zeta_{1\times1}(s)=\frac{1}{\pi^{2s}}\sum_{k_1,k_2\geqslant1}\frac{1}{(k_1^2+k_2^2)^s}=2\zeta_{\Delta}(s)+2^{-s}\zeta_{1}(s).
\end{equation*}
The spectral zeta function on the unit interval with Dirichlet boundary conditions is related to the Riemann zeta function by $\zeta_1(s)=(2/\pi^{2s})\zeta(2s)$ with special values in $0$, $\zeta(0)=-1/2$ and $\zeta'(0)=-(1/2)\log(2\pi)$. Thus we have that $\zeta'_1(0)=-2\log{2}$. By differentiating the above and evaluating in $s=0$, we get
\begin{equation}
\label{zeta11}
\zeta'_{1\times1}(0)=2\zeta'_{\Delta}(0)-\log{2}.
\end{equation}
Putting (\ref{log(tau(QAD))}), (\ref{L(n,n)}), (\ref{zeta11}) together and writing the derivative of the spectral zeta function in $0$ in terms of the regularized determinant of the Laplace-Beltrami operator on the right-angled isoscele unit triangle with Dirichlet boundary conditions, that is
\begin{equation*}
\zeta'_{\Delta}(0)=-\log{\textnormal{det}^\ast\Delta_{\Delta}},
\end{equation*}
gives the asymptotic behaviour of the number of spanning trees on the quartered Aztec diamond, namely
\begin{equation*}
\log(\tau(\textnormal{QAD}_n))=\frac{2G}{\pi}n^2-\log(2+\sqrt{2})n-\frac{3}{4}\log{n}+\log{\textnormal{det}^\ast\Delta_{\Delta}}+\frac{23}{8}\log{2}+o(1)
\end{equation*}
as $n\rightarrow\infty$.
\par\vspace{\baselineskip}
\noindent
\textbf{Acknowledgements:} The author thanks Anders Karlsson for suggesting this problem to her and for valuable discussions. She also thanks Larry Glasser for pointing reference \cite{MR2999720} to her. The author acknowledges support from the Swiss NSF grant $200021\_132528/1$.

\nocite{*}
\bibliographystyle{plain}
\bibliography{bibliography}

\begin{thebibliography}{10}

\bibitem{MR1271140}
Norman Biggs.
\newblock {\em Algebraic graph theory}.
\newblock Cambridge Mathematical Library. Cambridge University Press,
  Cambridge, second edition, 1993.

\bibitem{chinta2010zeta}
Gautam Chinta, Jay Jorgenson, and Anders Karlsson.
\newblock Zeta functions, heat kernels, and spectral asymptotics on
  degenerating families of discrete tori.
\newblock {\em Nagoya Math. J.}, 198:121--172, 2010.

\bibitem{chinta2011complexity}
Gautam Chinta, Jay Jorgenson, and Anders Karlsson.
\newblock Complexity and heights of tori.
\newblock In {\em Dynamical systems and group actions}, volume 567 of {\em
  Contemp. Math.}, pages 89--98. Amer. Math. Soc., Providence, RI, 2012.

\bibitem{MR2393253}
Mihai Ciucu.
\newblock Symmetry classes of spanning trees of {A}ztec diamonds and perfect
  matchings of odd squares with a unit hole.
\newblock {\em J. Algebraic Combin.}, 27(4):493--538, 2008.

\bibitem{MR952941}
Bertrand Duplantier and Fran{\c{c}}ois David.
\newblock Exact partition functions and correlation functions of multiple
  {H}amiltonian walks on the {M}anhattan lattice.
\newblock {\em J. Statist. Phys.}, 51(3-4):327--434, 1988.

\bibitem{MR2999720}
M.~L. Glasser.
\newblock A note on a hyper-cubic {M}ahler measure and associated {B}essel
  integral.
\newblock {\em J. Phys. A}, 45(49):494002, 4, 2012.

\bibitem{MR1965463}
G.~S. Joyce.
\newblock Singular behaviour of the lattice {G}reen function for the
  {$d$}-dimensional hypercubic lattice.
\newblock {\em J. Phys. A}, 36(4):911--921, 2003.

\bibitem{MR1862771}
G.~S. Joyce and I.~J. Zucker.
\newblock Evaluation of the {W}atson integral and associated logarithmic
  integral for the {$d$}-dimensional hypercubic lattice.
\newblock {\em J. Phys. A}, 34(36):7349--7354, 2001.

\bibitem{MR1819995}
Richard Kenyon.
\newblock The asymptotic determinant of the discrete {L}aplacian.
\newblock {\em Acta Math.}, 185(2):239--286, 2000.

\bibitem{MR1756162}
Richard~W. Kenyon, James~G. Propp, and David~B. Wilson.
\newblock Trees and matchings.
\newblock {\em Electron. J. Combin.}, 7:Research Paper 25, 34 pp. (electronic),
  2000.

\bibitem{louis2015asymptotics}
Justine Louis.
\newblock Asymptotics for the number of spanning trees in circulant graphs and
  degenerating $d$-dimensional discrete tori.
\newblock {\em Annals of Combinatorics}, pages 1--31, 2015.

\bibitem{louis2015formula}
Justine Louis.
\newblock A formula for the number of spanning trees in circulant graphs with
  non-fixed generators and discrete tori.
\newblock {\em arXiv preprint arXiv:1312.4389, accepted in Bulletin of the
  Australian Mathematical Society}, 2015.

\bibitem{MR3312509}
Nicolai Reshetikhin and Boris Vertman.
\newblock Combinatorial quantum field theory and gluing formula for
  determinants.
\newblock {\em Lett. Math. Phys.}, 105(3):309--340, 2015.

\end{thebibliography}

\end{document}